\documentclass[12pt]{amsart}
\usepackage[utf8]{inputenc}
\usepackage{gensymb}
\usepackage{amsmath}
\usepackage{enumitem}
\usepackage{amssymb}
\usepackage{amsthm}
\usepackage{graphicx}
\usepackage{tabularx}
\newtheorem*{Problem}{Q}
\newtheorem{Main}{Theorem}
\newtheorem{Main2}[Main]{Theorem}

\newtheorem{Theorem 2}[Main]{Theorem}
\newtheorem*{Question1}{Q1}
\newtheorem*{Question2}{Q2}
\newtheorem*{Question3}{Q3}
\newtheorem*{Question4}{Q4}
\newtheorem*{Question5}{Q5}

\newtheorem{Lemma 1}{Lemma}
\newtheorem{Lemma 2}[Lemma 1]{Lemma}
\newtheorem{Lemma 3}[Lemma 1]{Lemma}
\newtheorem{LemmaMatomaki}{Lemma}
\newtheorem{LemmaMatomaki2}[LemmaMatomaki]{Lemma}
\newtheorem{NewLemma3}[LemmaMatomaki]{Lemma}
\newtheorem{NewLemma4}[LemmaMatomaki]{Lemma}
\newtheorem{Corollary}{Corollary}
\usepackage{enumitem}
\usepackage{hyperref}
\usepackage{graphicx}
\makeatletter
\@namedef{subjclassname@2020}{%
	\textup{2020} Mathematics Subject Classification}
\makeatother

\date{\today}

\begin{document}

	
	\title{ON POWER VALUES OF SUM OF DIVISORS FUNCTION IN ARITHMETIC PROGRESSIONS}
	\author[S.T.Somu]{Sai Teja Somu}
	\address{Aspireal Technologies Pvt. Limited\\ Hyderabad,
		 India}
	\email{somuteja@gmail.com}
	
	\author[V.Mishra]{Vidyanshu Mishra}
	\address{Department of Applied Mathematics\\ Delhi Technological University\\
		Delhi, India}
	\email{vidyanshurpvvyv@gmail.com}
	
	\begin{abstract}
		Let $a\geq 1, b\geq 0$ and $k\geq 2$ be any given integers. It has been proven that there exist infinitely many natural numbers $m$ such that sum of divisors of $m$ is a perfect $k$th power. We try to generalize this result when the values of $m$ belong to any given infinite arithmetic progression $an+b$. We prove if $a$ is relatively prime to $b$ and order of $b$ modulo $a$ is relatively prime to $k$ then there exist infinitely many natural numbers $n$ such that sum of divisors of $an+b$ is a perfect $k$th power. We also prove that, in general,  either sum of divisors of $an+b$ is not a perfect $k$th power for any natural number $n$ or sum of divisors of $an+b$ is a perfect $k$th power for infinitely many natural numbers $n$.   
	\end{abstract}

	
	\maketitle
	\section{Introduction}

	Sum of divisors function (denoted by $\sigma$) is an important function in number theory. Perfect numbers, natural numbers $n$ satisfying $\sigma(n)=2n$, have been studied as early as 300 BC.  Euclid showed in \cite{Euclid} that $\frac{q(q+1)}{2}$ is an even perfect number whenever $q$ is a prime number of the form $2^p-1$, for natural number $p$.  Euler (in \cite{Euler}) proved that all even perfect numbers are of this form. It is still not known whether there are infinitely many perfect numbers and whether there exists an odd perfect number.      
	
	 The problem of finding natural numbers $n$ such that $\sigma(n)$ is a perfect $k$th power has been considered by 17th century mathematicians like Fermat, Wallis, and Frenicle. Fermat in 1657, asked two problems: (i) Find a cube $m=n^3$ for which $\sigma(m)$  is a perfect square (ii) Find a square $m=n^2$ for which $\sigma(m)$ is a perfect cube 
	 (see Chapter 2, p.54 of \cite{Dickson} for detailed history). 
	 
	 In \cite{Beukers} (see also \cite{Freiberg}, \cite{Zmija}) it has been proven that for any natural number $k\geq 2$, $\sigma(m)$ is a perfect $k$th power for infinitely many natural numbers $m$. There are several sequences in Neil Sloane's online encyclopedia of integer sequences \cite{Sloane} dedicated to this problem, namely A006532 ($\sigma(m)$ is perfect square), A020477 ($\sigma(m)$ is perfect cube), A019422 ($\sigma(m)$ is perfect fourth power), A019423 ($\sigma(m)$ is perfect fifth power), A019424 ($\sigma(m)$ is perfect sixth power), A048257
	 ($\sigma(m)$ is perfect seventh power) and A048258 ($\sigma(m)$ is perfect eighth power). 
	 
	 In this paper, we investigate this problem when the values of $m$ belong to any given infinite arithmetic progression $\{an+b: n\in \mathbb{N}\}$ (here $\mathbb{N}=\{1,2,3,\ldots\}$). We study the following problem.
	
	\begin{Problem}
		For any given integers $a\geq 1,b\geq 0,k\geq 2$ do there exist infinitely many natural numbers $n$ such that $\sigma(an+b)$ is a perfect $k$th power?
	\end{Problem}
	
	In the next section, we give an answer to the above problem when $a,b$ are relatively prime and order of $b$ modulo $a$ and $k$ are relatively prime. We prove the following theorem.
	
	\begin{Main}
		For any given relatively prime natural numbers $a,b$ and $k\geq 2$, if order of $b$ modulo $a$ is relatively prime to $k$, then there exist infinitely many natural numbers $n$ such that $n\equiv b \mod a$ and $\sigma(n)$ is a perfect $k$th power. 
	\end{Main}
	
	For general $b$, we prove the following theorem.
	\begin{Main2}
		For any given integers $a\geq 1,b\geq 0$ and $k\geq 2$, either $\sigma(an+b)$ is not a perfect $k$th power for any natural number $n$ or $\sigma(an+b)$ is a perfect $k$th power for infinitely many natural numbers $n$.
	\end{Main2}
	
	In the final section, we give few open questions for further research.  
	
	\section{Proofs of Theorem 1 and Theorem 2}
	We require some lemmas in order to prove Theorem 1 and Theorem 2. We will use a group theoretic lemma and a number theoretic lemma mentioned in \cite{Matomaki}.
	
	Let $\mathbb{P}$ be set of primes. For relatively prime positive integers $a,b$ define, $$\pi_{a,b}(x,y) = \left|\left\{\frac{x}{2}\leq p \leq x: p\in \mathbb{P},p\equiv b \mod a \text{ and }P(p+1)\leq y\right\}\right|,$$ where $P(n)$ denotes largest prime factor of $n$. 
	
	The next lemma is essentially \cite[Lemma 2]{Matomaki}, with a slight difference. Since there is a minor difference we give a proof of the lemma here.
	
	\begin{LemmaMatomaki}
		Let $a,b$ be fixed coprime integers with $b>0$. For any $\alpha>\frac{1}{2}$, there exist $\gamma(\alpha)>0$ and $x_1(\alpha,a)$ such that $$\pi_{a,b}(x,x^{\alpha})\geq \frac{\gamma(\alpha)x}{\phi(a)\log x}$$ for all $x\geq x_1(\alpha,a)$.
	\end{LemmaMatomaki}
	\begin{proof}
	 Since proving for $\frac{1}{2}<\alpha<\frac{2}{3}$ suffices, we can assume $\alpha<\frac{2}{3}$. Choose a positive $\epsilon =\epsilon(\alpha)<\alpha-\frac{1}{2}$. Now if $p\leq x-1$ is such that $p=-1+qk$ for a prime $q\in [x^{1-\alpha},x^{\frac{1}{2}-\epsilon}]$, then as $p+1=qk$, $q\leq x^{\frac{1}{2}-\epsilon}<x^{\alpha}$ and $k=\frac{p+1}{q}\leq \frac{x}{x^{1-\alpha}}\leq x^{\alpha}$ we have $P(p+1)\leq x^{\alpha}$. 
	 Each $\frac{x}{2}\leq p\leq x-1$ can have at most two such representations (as if there are three different representations $p=-1+q_1k_1=-1+q_2k_2=-1+q_3k_3$ then $q_1q_2q_3|p+1$ which cannot be true as $p+1\leq x<x^{3(1-\alpha)} \leq q_1q_2q_3$). Hence \[\pi_{a,b}(x,x^{\alpha})\geq  \frac{1}{2}\sum_{\substack{{x^{1-\alpha}<q<x^{\frac{1}{2}-\epsilon}}\\ {q\in \mathbb{P}}}}\sum_{\substack{{\frac{x}{2}\leq p\leq x-1}\\{p\equiv -1 \mod q}\\{p\equiv b \mod a}}}1  .\]
	 When $x$ is large enough, then the congruence conditions can be combined into a single congruence $\mod aq$, so, by the Bombieri-Vinogradov theorem, 
	 \begin{align*}
	 \pi_{a,b}(x,x^{\alpha})&\geq \frac{1+o(1)}{2\phi(a)}\sum_{\substack{{x^{1-\alpha}<q<x^{\frac{1}{2}-\epsilon}}\\{q\in \mathbb{P}}}}\left(\frac{x}{2\phi(q)\log x}+O(x(\log x)^{-10})\right)
	 \\&\geq \frac{\log (\frac{1}{2}-\epsilon)x}{8(1-\alpha)\phi(a)\log x}
	 \end{align*}
	 for every large enough $x$.
	\end{proof}
	Let $\lambda(G)$ denote maximal order of the elements of a group $G$, and let $\Omega(h)$ denote the number of prime number divisors of the natural number $h$, counted with multiplicities.   
	We have the following lemma from \cite[Lemma 6]{Matomaki}.
	\begin{LemmaMatomaki2}
		For any multiplicative group $G$ write $$s(G)=\lceil5\lambda(G)^2\Omega(|G|)\log {(3\lambda(G)\Omega(|G|))} \rceil .$$Let $A$ be a sequence of length $n$ consisting of nonidentity elements of $G$. Then there exists a nontrivial subgroup $H\leq G$ such that if $n\geq s(G)$, then, for every $h\in H$, $A\cap H$ has a subsequence whose product is equal to $h$.
	\end{LemmaMatomaki2}
	\begin{proof}
		See Lemma 6 of \cite{Matomaki}.
	\end{proof}
	
	\begin{NewLemma3}
		Let $k\geq 2,t$ be relatively prime positive integers and $S$ be any nonempty finite set of primes. Let $A$ be any finite set of natural numbers such that all the prime factors of $A$ belong to $S$. If $|A|\geq \lceil 5k^2t^2(|S|\Omega(k)+\Omega(t))\log (3kt(|S|\Omega(k)+\Omega(t))) \rceil$ then there exists a nonempty subset $A'$ of $A$ such that $\prod_{a\in A'}a$ is a perfect $k$th power and $|A'|\equiv 1\mod t$.
	\end{NewLemma3}
	\begin{proof}
		Let $n_1,\ldots, n_{|A|}$ be distinct elements of $A$ and let $p_1,\ldots,p_{|S|}$ be distinct elements of $S$. Let $G=(\mathbb{Z}/k\mathbb{Z})^{|S|}\times (\mathbb{Z}/t\mathbb{Z})$. As all the prime factors of $n_i$ are in $S$ there exist nonnegative integers $a_{i,1},\ldots, a_{i,|S|}$ such that $n_i=p_1^{a_{i,1}}\cdots p_{|S|}^{a_{i,|S|}}$. Now define $e_i\in G$ in the following way: $$e_i=(a_{i,1}\mod k,\cdots, a_{i,|S|}\mod k, 1 \mod t),$$ for $1\leq i \leq |A|$. Now for group $G$, we have maximal order $\lambda(G)=kt$, and the number of elements equal to $|G|=k^{|S|}t$, $\Omega(|G|)=|S|\Omega(k)+\Omega(t)$. Now, $s(G)$ given in Lemma 2 is equal to, $s(G)=\lceil 5k^2t^2 (|S|\Omega(k)+\Omega(t))\log(3kt(|S|\Omega(k)+\Omega(t)))\rceil$.
		If any $e_i$ is identity element then  $n_i$ is a perfect $k$th power and $t=1$, which implies there is a subset $A'=\{n_i\}$ of $A$ such that $\prod_{a\in A'}a = n_i$ is perfect $k$th power and $|A'|\equiv 1\mod t$. Hence we can assume $e_i$ is not equal to identity for all $i\in \{1,\ldots,|S|\}$. 
		 
		 Since length of the sequence, $|A|\geq s(G)$, from Lemma 2, there exists a nontrivial subgroup $H\leq G$ such that for every $h\in H$, $\{e_1,\cdots,e_{|A|}\}\cap H$ has a subsequence whose product is equal to $h$.
		
		Now let $H$ be such a subgroup, clearly $\{e_1,\cdots,e_{|A|}\}\cap H$ is nonempty. Let $e\in \{e_1,\cdots,e_{|A|}\}\cap H$. As $k,t$ are relatively prime, there exists a positive integer solution to the congruences $n\equiv 0 \mod k$, and $n\equiv 1 \mod t$. Let $n$ be one solution. Now as $$n.e=(0\mod k,\ldots, 0 \mod k, 1 \mod t)\in H$$ , from Lemma 2, there exists a subsequence $e_{m_1},\ldots, e_{m_r}$ such that $$\sum_{i=1}^{r}e_{m_i}=(0 \mod k,\ldots, 0 \mod k, 1 \mod t).$$
		Let $A'=\{n_{m_1},\ldots, n_{m_r}\}\subset A$. From the definition of $e_i$ and $\sum_{i=1}^{r}e_{m_i}=(0 \mod k,\ldots, 0 \mod k, 1 \mod t)$ it follows that $r=|A'|\equiv 1 \mod t$ and all the exponents in the prime factorization of $\prod_{a\in A'}a$ are multiples of $k$, hence $\prod_{a\in A'}a$ is a perfect $k$th power.
	\end{proof}
\begin{NewLemma4}
Let $a,b,k$ be positive integers such that $a,b$ are relatively prime and $k\geq 2$ is relatively prime with order of $b$ modulo $a$. There exists a positive real number $x_0(a)$ such that for all $x\geq x_0(a)$ there exists a natural number $n$ such that all the prime factors of $n$ lie in the interval $[\frac{x}{2},x]$, $n\equiv b \mod a$, and $\sigma(n)$ is perfect $k$th power.		
\end{NewLemma4}
\begin{proof}
Let $$A=\left\{p+1: \text{$p$ prime}, \frac{x}{2}\leq p\leq x, p\equiv b \mod a, \text{ and } P(p+1)\leq x^{\frac{3}{5}} \right\},$$ and $S=\{p: p\text{ prime }, p\leq x^{\frac{3}{5}}\}$. Let order of $b$ modulo $a$ be equal to $t$. From prime number theorem, we have $|S|=O\left(\frac{x^{\frac{3}{5}}}{\log x}\right)$, and therefore $\lceil 5k^2t^2(|S|\Omega(k)+\Omega(t))\log (3kt(|S|\Omega(k)+\Omega(t))) \rceil = O(x^{\frac{3}{5}})$. From Lemma 1, for sufficiently large $x$, we have $|A|\geq \frac{\gamma(\frac{3}{5})x}{\phi(a)\log x}$. Hence for sufficiently large $x$, $|A|\geq \lceil 5k^2t^2(|S|\Omega(k)+\Omega(t))\log (3kt(|S|\Omega(k)+\Omega(t))) \rceil$. Applying Lemma 3 we get that there exists a subset $A'$ of $A$ such that $|A'|\equiv 1 \mod t$ and $\prod_{p+1\in A'}(p+1)$ is a perfect $k$th power. Let $n=\prod_{p+1\in A'}p$, we have $\sigma(n)=\prod_{p+1\in A'}(p+1)$ as a perfect $k$th power. Since from the definition of $A$ we have $p+1\in A'$ implies $p\equiv b \mod a$ we can conclude that $n=\prod_{p+1\in A'}p\equiv b^{|A'|}\equiv b \mod a$, which completes the proof of the lemma.   	
\end{proof}

	We can now prove Theorem 1 and Theorem 2.
	\subsection{Proof of Theorem 1}
	\begin{proof}
		Assume for a contradiction that there are only finitely many $n$ such that $n\equiv b \mod a$ and  $\sigma(n)$ is a perfect $k$th power. Then there exists an $x_1>0$ such that for all $n\equiv b \mod a$ and $n\geq x_1$, $\sigma(n)$ is not a perfect $k$th power. This implies that there cannot be any $n$ whose prime factors lie in the range $[\frac{x}{2},x]$ for $x=\max\{2x_0(a),2x_1\}$, is congruent to $b$ modulo $a$, and whose sum of divisors is a perfect $k$th power, contradicting Lemma 4.   
	\end{proof}
    We have a corollary for Theorem 1.
	\begin{Corollary}
		For any given natural numbers $a$ and $k\geq 2$ there exist infinitely many natural numbers $n$ such that $n\equiv 1 \mod a$ and $\sigma(n)$ is perfect $k$th power.
	\end{Corollary}
	\begin{proof}
		As order of $b=1$ modulo $a$ equal to $1$, is relatively prime to $k$, the corollary follows from  Theorem 1. 
	\end{proof}
	\subsection{Proof of Theorem 2}
	\begin{proof}
		Suppose for a contradiction that there are nonzero finitely many natural numbers of the form $an+b$ such that $\sigma(an+b)$ is a perfect $k$th power, then let $an_m+b$ be largest such natural number. From Lemma 4, for $b=1$, there exists a natural number $r\equiv 1 \mod a$ such that all the prime factors of $r$ lie in $[\frac{x}{2},x]$ for $x=\max\{x_0(a),4(an_m+b)\}$ and $\sigma(r)$ is perfect $k$th power. Since all prime factors of $m$ are greater than $an_m+b$, we have $an_m+b$ is relatively prime to $r$. We have $(an_m+b)r>an_m+b$,  $\sigma((an_m+b)r)=\sigma(an_m+b)\sigma(r)$ is a perfect $k$th power and $(an_m+b)r\equiv b \mod a$, contradicting maximality of $an_m+b$.   
	\end{proof}
	
	\section{Future prospects}
	We ask five questions Q1 to Q5 in this section.
	
	Given $a,b$ and $k\geq 2$, Theorem 2 leaves us with two possibilities, either $\sigma(an+b)$ is perfect $k$th power for infinitely many $n$ or $\sigma(an+b)$ is never a perfect $k$th power. This implies that if there exists a natural number $n$ such that $\sigma(an+b)$ is perfect $k$th power then there are infinitely many $n$ such that $\sigma(an+b)$ is a perfect $k$th power. So we ask the following question.
	
	\begin{Question1}
		Do there exist integers $a\geq 1,b\geq 0,k\geq 2$ such that $\sigma(an+b)$ is not a perfect $k$th number for any natural number $n$? 
	\end{Question1}
	If the answer to Q1 is no, then from Theorem 2, it follows that for all integers $a\geq 1,b\geq 0,k\geq 2$ there exist infinitely many $n$ such that $\sigma(an+b)$ is a perfect $k$th power.
	
	For natural numbers $a$, $k\geq 2$ and integer $b$ such that  $0\leq b\leq a-1$, let $N(x;a,b,k)$ denote number of natural numbers $n\leq x$ such that $n\equiv a \mod b$ and $\sigma(n)$ is perfect $k$th power. 
	
	\begin{Question2}
		Given natural numbers $a,0\leq b_1 ,b_2\leq a-1,k\geq 2$ such that $N(x;a,b_1,k)>0$, $N(x;a,b_2,k)>0$ for some real number $x$. Is it true that 
		$$\liminf_{x\rightarrow \infty} \frac{N(x;a,b_1,k)}{N(x;a,b_2,k)}>0?$$
	\end{Question2}

	We will now ask a stronger version of the above question. Are perfect $k$th power values of $\sigma(n)$ equally distributed among residue classes modulo $a$?
	\begin{Question3}
		Given natural numbers $a,0\leq b_1 ,b_2\leq a-1,k\geq 2$, is it true that there exists a real number $x$ such that $N(x;a,b_1,k)>0$, $N(x;a,b_2,k)>0$ and  
		$$\lim_{x\rightarrow \infty} \frac{N(x;a,b_1,k)}{N(x;a,b_2,k)}=1?$$
	\end{Question3}
	
	Numerical evidence suggests that Q3 may not always be true, for example when we consider perfect square values of $\sigma(n)$ for $n\leq x$ and residue classes $n\equiv 0 \bmod 7$ and $n\equiv 1 \bmod 7$, we have the following data.  
	\begin{table}[h!]
		\centering
		\begin{tabularx}{0.8\textwidth} { 
				| >{\raggedright\arraybackslash}X 
				| >{\centering\arraybackslash}X 
				| >{\raggedleft\arraybackslash}X |  >{\raggedleft\arraybackslash}X    | }
			\hline
			$x$ & $N(x;7,0,2)$ & $N(x;7,1,2)$ & $\frac{N(x;7,0,2)}{N(x;7,1,2)}$\\
			\hline
			100 & 2 & 2 & 1.0000...\\
			\hline
			1000 & 16 & 5 & 3.2000...\\
			\hline
			10000 & 68 & 28 &  2.4285...\\
			\hline
			100000 & 307 & 147 & 2.0884... \\
			\hline
			1000000 & 1508 & 748 & 2.1299...\\
			\hline
			10000000 & 7562 & 3811  & 1.9842... \\
			\hline
			100000000 & 37652 & 18882 & 1.9940... \\
			\hline
		\end{tabularx}
		\caption{Comparison of $N(x;7,0,2)$ and $N(x;7,1,2)$}
	\end{table}
	
	Table 1 suggests that, there might be larger number of natural numbers $n$ of the form $7k$ compared to $7k+1$ for which $\sigma(n)$ is a perfect square. We ask the following question.
	
	\begin{Question4}
		Is it true that $$\limsup_{n\rightarrow \infty} \frac{N(x;7,0,2)}{N(x;7,1,2)}>1?$$
	\end{Question4}
	
	If the answer for question Q4 is yes, then the answer for Q3 is no.

	 Till now, we have been dealing with linear polynomials $an+b$. A natural generalization would be the following question.
	
	\begin{Question5}
		Let $p(n)$ be any given polynomial with nonnegative integer coefficients and $k\geq 2$ be any natural number. Is $\sigma(p(n))$ perfect $k$th power for infinitely many $n$?
	\end{Question5}
	
	Question 2.5 (when $p(n)=n^l$), Conjecture 2.6 (when $p(n)=n^2$), Conjecture 2.7 (when $p(n)=n^3$) of \cite{Beukers} are special cases of the above question.
	In Neil Sloane’s online encyclopedia of integer sequences \cite{Sloane} we find few sequences related to this problem, namely, A008847 ($\sigma(n^2)$ is a perfect square, $p(n)=n^2$, $k=2$), A008850
	($\sigma(n^2)$ is a perfect cube, $p(n)=n^2$, $k=3$) and A008849 ($\sigma(n^3)$ is a perfect square, $p(n)=n^3$, $k=2$).

\end{document}